\documentclass[a4paper,reqno,12pt]{amsart}
\usepackage[dvips]{graphicx}
\usepackage{graphics}
\usepackage{amsthm, amssymb, amsmath, amsfonts}

\newtheorem{thm}{Theorem}[section]
\newtheorem{lem}{Lemma}[section]
\newtheorem{prop}{Proposition}[section]
\newtheorem{cor}{Corollary}[section]
\newtheorem{defn}{Definition}[section]
\newtheorem{ex}{Example}[section]

\newtheorem{rmk}{Remark}[section]
\newcommand{\Real}{{\mathbb R}}
\newcommand{\Natural}{{\mathbb N}}
\newcommand{\ds}{\displaystyle}
\newcommand{\To}{\longrightarrow}
\def\1{\^{\i}}
\def\2{\u{a}}
\def\3{\c{s}}
\def\4{\^{a}}
\def\5{\c{t}}

\def\e{\epsilon}

\def\<{\langle}
\def\>{\rangle}
\DeclareMathOperator*\dom{dom}

\DeclareMathOperator*\cl{cl}

\DeclareMathOperator*\inte{int}

\begin{document}
\title{Generalized monotone operators on dense sets}
\author{ Szil\'ard L\' aszl\' o$^*$ \hspace{2cm} Adrian Viorel$^\dagger$}
\thanks{$^*$This work was supported by a grant of the Romanian Ministry of Education, CNCS - UEFISCDI, project number PN-II-RU-PD-2012-3 -0166.}
\thanks{$^\dagger$This research was supported by a grant of the Romanian
National Authority for Scientific Research CNCS - UEFISCDI, project number
PN-II-PCE-2011-3-0094.}
\address{S. L\' aszl\' o, Department of Mathematics, Technical University of Cluj Napoca,
Str. Memorandumului nr. 28, 400114 Cluj-Napoca, Romania.}
\email{szilard.laszlo@math.utcluj.ro}
\address{A. Viorel,Department of Mathematics, Technical University of Cluj Napoca,
Str. Memorandumului nr. 28, 400114 Cluj-Napoca, Romania.}
\email{adrian.viorel@math.utcluj.ro}

\begin{abstract}
In the present work we show that the local generalized monotonicity of a lower semicontinuous set-valued operator on some certain type of dense sets ensures the global
generalized monotonicity of that operator. We achieve this goal gradually by showing at first that the lower semicontinuous set-valued functions of one
real variable, which are locally generalized monotone on a dense subsets of their domain are globally generalized monotone. Then, these results are  extended to the case  of set-valued operators  on arbitrary Banach spaces. We close this work with a  section on the global generalized convexity of a real valued function, which is obtained out of its local counterpart on some dense sets.\newline


\end{abstract}

\maketitle

\section{Introduction}

In this paper we provide sufficient conditions which ensure that a lower semicontinuous set-valued operator satisfying a generalized monotonicity property locally is actually globally generalized monotone. We are concerned on strictly quasimonotone, pseudomonotone as well as strictly pseudomonotone set-valued operators, and these results extend and improve the results obtained in the single-valued case in \cite{La} and \cite{KaPiLa}.

 We study generalized monotonicity properties on certain dense subsets that we call self segment-dense. By an example we show that this new concept differs from that of a segment-dense set introduced by Dinh The Luc \cite{DTL} in the context of densely quasimonotone, respectively densely pseudomonotone operators.

Two counterexamples circumscribe the area of our research. We show that the local (generalized) monotonicity of an upper semicontinuous set-valued operator does not imply its global counterpart (Example \ref{ex2}). Then we also provide an example of lower semicontinuous set-valued operator, which is locally quasimonotone on its domain but is not globally quasimonotone (Example \ref{ex1}). Hence, we deal with locally strictly quasimonotone, locally pseudomonotone, respectively locally strictly pseudomonotone lower semicontinuous set-valued operators.

The outline of the paper is the following. In the next section  we introduce the necessary apparatus that we need to obtain our results. We also introduce the notion of a self segment-dense set and we show by an example that this notion differs to the notion of segment-dense set introduced in \cite{DTL}. Further, we emphasize that the local generalized monotonicity on a self segment-dense set contains as particular cases the local monotonicity conditions considered in \cite{KaPiLa}, \cite{KaPiLa1} and \cite{La}.
Section 3 contains the main results. Our analysis starts with the one dimensional case ($X=\mathbb{R}$) in which we show that a lower semicontinuous set-valued operator that is locally generalized monotone on a dense subset of its domain must be globally generalized monotone. These results are then extended to the case of arbitrary Banach spaces using the concept of a self segment-dense subset. In Section 4 the previous results are applied to obtain the global generalized convexity of locally generalized convex functions under mild assumptions. Also an example of a locally quasiconvex continuously differentiable function which is not globally quasiconvex is given.

\section{Preliminaries}

In what follows $X$ denotes a real Banach space and  $X^{\ast }$ denotes the topological dual of $X.$ For $x\in X$ and $x^{\ast }\in X^{\ast }$ we denote by $\<x^{\ast },x\>$
the scalar $x^{\ast }(x)$.

We will also often  use the following notations for the open, respectively
closed, line segments in $X$ with the endpoints $x$ and $y$
\begin{eqnarray*}
(x,y) &=&\big\{z\in X:z=x+t(y-x),\,t\in (0,1)\big\}, \\
\lbrack x,y] &=&\big\{z\in X:z=x+t(y-x),\,t\in \lbrack 0,1]\big\}.
\end{eqnarray*}%

For a non-empty set $D\subseteq X$, we denote by  $\inte(D)$ its interior and by $\cl(D)$ its closure. We say that $P\subseteq D$ is dense in $D$ if $D\subseteq \cl(P)$, and that $P\subseteq X$ is closed regarding $D$ if $\cl(P)\cap D=P\cap D.$

Let $T:X\rightrightarrows X^{\ast }$ be a set-valued operator. We denote by $D(T)=\{x\in X:T(x)\neq \emptyset \}$ its domain and by $R(T)=\displaystyle\bigcup_{x\in D(T)}{T(x)}$ its range. The graph of the operator $T$ is the set $G(T)=\{(x,x^{\ast })\in X\times X^{\ast }:\,x^{\ast }\in T(x)\}.$

Recall that $T$ is said to be \emph{upper semicontinuous} at $x\in D(T)$ if
for every open set $N\subseteq X^{\ast }$ containing $T(x)$, there exists a
neighborhood $M\subseteq X$ of $x$ such that $T(M)\subseteq N.$ $T$ is said
to be \emph{lower semicontinuous} at $x\in D(T)$ if for every open set $%
N\subseteq X^{\ast }$ satisfying $T(x)\cap N\neq \emptyset $, there exists a
neighborhood $M\subseteq X$ of $x$ such that for every $y\in M\cap D(T)$ one
has $T(y)\cap N\neq \emptyset .$ $T$ is upper semicontinuous (lower
semicontinuous) on $D(T)$ if it is upper semicontinuous (lower
semicontinuous) at every $x\in D(T).$

It can  easily  be observed that the definition of lower semicontinuity is equivalent to the following (see \cite{AF} Def. 1.4.2).

\begin{rmk}\rm\label{r21} $T$ is lower semicontinuous at $x\in D(T)$, if and only if, for
every sequence $({x_n})\subseteq D(T)$ such that $x_n\longrightarrow x$ and
for every $x^*\in T(x)$ there exists a sequence of elements $x^*_n \in T(x_n)
$ such that $x^*_{n}\longrightarrow x^*.$
\end{rmk}

Obviously, when $T$ is single-valued then upper semicontinuity and also lower semicontinuity become the usual notion of continuity.


\begin{defn}\label{d3.1} Let  $T:X\longrightarrow X^*$ be a single valued operator. We say that $T$ is hemicontinuous at $x\in X$, if for all $(t_n)_{n\in\Natural}\subset \Real,\,t_n{\longrightarrow} 0,\, {n\longrightarrow\infty}$ and $y\in X,$ we have $T(x+t_n y)\rightharpoonup^* T(x),\, n\longrightarrow \infty,$ where $"\rightharpoonup^*"$ denotes the convergence with respect to the $weak^*$ topology of $X^*$.
\end{defn}
\subsection{Local generalized monotonicities}
\noindent

In what follows we recall the definitions of several monotonicity concepts (see, for instance \cite{HKS,L-S,PQ}). The operator $T:X\rightrightarrows X^*$ is called:

\begin{enumerate}
\item[(1)] \emph{\ monotone}  if for all $(x,x^*),(y,y^*)\in G(T)$ one has the following
\begin{equation*}
\left\langle x^*-y^*,x-y\right\> \geq0;
\end{equation*}

\item[(2)] \emph{\ pseudomonotone} if for all $(x,x^*),(y,y^*)\in G(T),$ the following implication holds
\begin{equation*}
\<x^*,y-x\> \ge 0\Longrightarrow  \<y^*,y-x\>\ge 0,
\end{equation*}
or equivalently, for all $(x,x^*),(y,y^*)\in G(T),$ one has
\begin{equation*}
\<x^*,y-x\> > 0\Longrightarrow  \<y^*,y-x\>>0;
\end{equation*}

\item[(3)] \emph{strictly pseudomonotone} if for all $(x,x^*),(y,y^*)\in G(T),\, x\neq y,$ the following implication holds
\begin{equation*}
\<x^*,y-x\> \ge 0\Longrightarrow  \<y^*,y-x\>> 0;
\end{equation*}

\item[(4)] \emph{quasimonotone}  if for every $(x,x^*),(y,y^*)\in G(T)$ the following implication holds
\begin{equation*}
\<x^*,y-x\> > 0 \Longrightarrow  \<y^*,y-x\>\ge 0;
\end{equation*}

\item[(5)] \emph{strictly quasimonotone},  if $T$ is
quasimonotone, and for all $x,y\in D(T),\,x\neq y$ there exist $z\in(x,y)$
and $z^*\in T(z)$ such that$\<z^*,y-x\> \neq 0.$
\end{enumerate}

The relation among these concepts is shown bellow.
\begin{center}
\begin{tabular}{cc}
  $$ & $T \mbox{ is monotone} $\\
  $$ & $\Downarrow $\\
  $ T \mbox{ is strictly pseudomonotone}\Longrightarrow$ & $T \mbox{ is pseudomonotone}$\\
  $\Downarrow$ & $\Downarrow$\\
  $T \mbox{  is strictly quasimonotone}\Longrightarrow$ & $T \mbox{ is quasimonotone}.$\\
\end{tabular}
\end{center}

The concepts of local monotonicity can be defined as follows.
\begin{defn}
 We say that the operator $T$ is locally monotone, (respectively, locally pseudomonotone, locally strict pseudomonotone, locally quasimonotone, locally strict quasimonotone), on its domain $D(T)$, if every $x\in D(T)$ admits   an open neighborhood $U_x\subseteq X$ such that the restriction of the operator $T$ on $U_x\cap D(T),\, T\big|_{U_x\cap D(T)}$ is  monotone, (respectively, pseudomonotone, strictly pseudomonotone, quasimonotone, strictly quasimonotone).
\end{defn}

According to \cite{KaPiLa1}, respectively \cite{La}, local monotonicity, respectively local generalized monotonicity of a real valued function of one real variable in general  implies its global counterpart. More precisely the following hold.

\begin{prop}\label{p21} Let $I\subseteq\mathbb{R}$ be an open interval and let $f:I\longrightarrow\mathbb{R}$ be a function.
\begin{itemize}
\item[(a)]{\rm (see Lemma 2.1, \cite{KaPiLa1})}  If $f$ is locally increasing on $I$, then $f$ is globally increasing on $I$.
\item[(b)]{\rm (see Theorem 2.2, \cite{La})} If $f$ is locally strictly quasimonotone on $I$, then $f$ is globally strictly quasimonotone on $I$.
\item[(c)]{\rm(see Theorem 2.3, \cite{La})}  If $f$ is locally pseudomonotone on $I$, then $f$ is globally pseudomonotone on $I$.
\item[(d)]{\rm(see Theorem 2.4, \cite{La})} If $f$ is locally strictly pseudomonotone on $I$, then $f$ is globally strictly pseudomonotone on $I$.
\end{itemize}
\end{prop}

However, local quasimonotonicity does not imply global quasimonotonicity as the next example shows.

\begin{ex}\rm(Example 2.6, \cite{La})
Let $f:\Real\longrightarrow\Real,\,$
$$
f(x)=\left\{
\begin{array}{lll}
-x-1,\,\mbox{if\,}\, x<-1\\
0,\,\mbox{if\,}\, x\in[-1,1]\\
-x+1,\,\mbox{if\,}\,x>1.\\
\end{array}
\right.
$$
We can see that $f$ is continuous and locally quasimonotone on $\Real$ but $f$ is not globally quasimonotone.
\end{ex}

Next we extend  the statements $(a)-(d)$  in Proposition \ref{p21} to the case of set-valued operators. We need the following notion.

Let $Z$ and $Y$ be two arbitrary sets. A single valued selection of the set-valued map $F:Z\rightrightarrows Y$ is the single valued map $f:Z\longrightarrow Y$
satisfying $f(z)\in F(z)$ for all $z\in Z.$


\begin{prop}\label{p31} Let $I\subseteq {\mathbb{R}}$ be an interval and let $F:I\rightrightarrows {\mathbb{R}}$ be a set-valued function. Assume that $F$
is locally strictly quasimonotone, (respectively locally monotone, locally pseudomonotone, locally strictly pseudomonotone) on $I$.
Then $F$ is strictly quasimonotone, (respectively monotone,  pseudomonotone, strictly pseudomonotone) on $I$.
\end{prop}

\begin{proof}
We treat the case when $F$ is strictly quasimonotone, the other cases can be treated similarly. Observe first that due to the local strict quasimonotonicity property of $F$, for every $x,y\in I,\,x\neq y,$ there exists $z\in (x,y)$ and $z^*\in F(z)$ such that $\<z^*,y-x\>\neq 0.$

Let $x,y\in \inte(I).$ Suppose the contrary, that is, $\<x^*,y-x\>>0$ for
some $x^*\in F(x)$ and $\<y^*,y-x\><0,$ for some $y^*\in F(y).$

Let $f$ be a single valued selection of $F$, which is also locally strictly quasimonotone,
and assume that $f(x)=x^*$ and $f(y)=y^*.$ According to Proposition \ref{p31} $(b)$ we get that $f$ is globally strictly quasimonotone on $I$, hence $\<f(y),y-x\>=\<y^*,y-x\>\ge 0,$ contradiction. Thus, $F$ is globally strictly quasimonotone
on $\inte(I).$

If $I$ is closed from left assume that $x$ is the left endpoint of $I$ and $%
\<x^*,y-x\>> 0$ for some $x^*\in F(x)$ and $y\in I.$ Then obviously $x^*>0.$
Assume further that there exists $y^*\in F(y)$ such that $\<y^*,y-x\><0,$
i.e. $y^*<0.$ Since $F$ is locally strictly quasimonotone on $I$, there exists an open
neighborhood $U_x$ of $x$, such that $F|_{U_x\cap I}$ is strictly quasimonotone, which in
particular shows that for some $u\in \inte(I)\cap U_x$ there exists $u^*\in
F(u),\, u^*>0.$ On the other hand, it can be deduced in the same manner,
that there exists an open neighborhood $U_y$ of $y$, such that $F|_{U_y\cap I}
$ is strictly quasimonotone, which in particular shows that for some $v\in \inte(I)\cap U_y$
there exists $v^*\in F(v),\, v^*<0.$ Obviously we can assume $u<v.$ But
then, $\<u^*,v-u\>>0$ and $\<v^*,v-u\><0$ which is in contradiction with the
fact that $F$ is strictly quasimonotone on $\inte(I).$

The case when $I$ is closed from right can be treated similarly and we omit it.
\end{proof}

\begin{rmk}\rm
\label{r31} Note that for the case when $F$ is locally monotone, locally pseudomonotone,  respectively locally strictly pseudomonotone, one may use Proposition \ref{p31} $(a),\,(c)$ and $(d)$ respectively.
\end{rmk}

The next example shows, that locally quasimonotone set-valued operators are not globally quasimonotone in general. Even more, our operator is also lower semicontinuous, hence searching for conditions that ensure global quasimonotonicity of a lower semicontinuous set-valued operator, based on its local quasimonotonicity, is meaningless.
\begin{ex}\rm
\label{ex1} Consider the operator
\begin{equation*}
F:\mathbb{R} \rightrightarrows\mathbb{R} ,\, F(x):=\left\{
\begin{array}{lll}
[0,-x-1],\mbox{ if } x<-1, &  &  \\
\, 0,\mbox{ if } x\in[-1,1], &  &  \\
\, [-x+1,0],\mbox{ if } x>1.
\end{array}
\right.
\end{equation*}
Then $F$ is lower semicontinuous and locally {quasimonotone} on ${\mathbb{R}}$,
but is not globally {quasimonotone}.

It is easy to check that $F$ is locally {quasimonotone} and lower semicontinuous on ${\mathbb{R}}.$ On the other hand, for $x=-2,\,y=2$ and $x^{\ast }=1\in F(x)$
we have $\<x^{\ast },y-x\>=4>0,$ but for $y^{\ast }=-1\in F(y)$ we have $\<y^{\ast },y-x\>=-4<0,$ which shows that $F$ is not {quasimonotone}.
\end{ex}

Next we define the local generalized monotonicity of an operator on a subset of its domain.

\begin{defn}
\label{d2.1} Let $T:X\rightrightarrows X^*$ be a set-valued operator and let $D\subseteq D(T).$ We say that $T$ is locally quasimonotone,
(respectively, locally monotone, locally pseudomonotone, locally strictly pseudomonotone, locally strictly quasimonotone), on $D$, if every $x\in D$ admits an open neighborhood $U_x$, such that the restriction $T|_{U_x\cap D}$ is {quasimonotone}, (respectively, {monotone}, {pseudomonotone}, {strictly pseudomonotone}, {strictly quasimonotone}).
\end{defn}

\begin{rmk}\label{r2.2}\rm
\textrm{ There are other ways to define local generalized monotonicity of an operator on a subset of its domain. Indeed, one may
assume that: }

\begin{itemize}
\item[(i)] \textrm{every $x\in D(T)$ admits an open neighborhood $U_x$, such
that the restriction $T|_{U_x\cap D}$ has the appropriate (generalized)
monotonicity property, 
respectively, }

\item[(ii)] \textrm{every $x\in D$ admits an open neighborhood $U_x$, such
that the restriction $T|_{U_x\cap D(T)}$ has the appropriate (generalized)
monotonicity property. }
\end{itemize}

\textrm{However, it can be observed that the condition given in Definition \ref{d2.1} is the weakest among these conditions. Note that if $D(T)$ and
also $D$ are open, then the above definitions coincide. }
\end{rmk}

\subsection{Some global generalized monotonicity results for single-valued  operators}
\noindent


Having Definition \ref{d2.1} in mind, an important question is related to which properties must the subset $D\subseteq D(T)$ have such that the local generalized monotonicity on $D$ implies the global monotonicity of the same kind. The following results concerning single-valued operators, established in \cite{KaPiLa1}, respectively \cite{La}, will serve as a starting point for our investigations in Section 3.

\begin{prop}\label{p}\emph{(Theorem 3.4, \cite{KaPiLa1})}   Let $X$ be a real Banach space and let $D\subseteq X$ be open and convex. Let $C\subseteq D$  be a set closed regarding $D,$ with empty interior, satisfying the following condition
$$ \forall\,x,y\in D\setminus C \mbox{ the set } [x,y]\cap C \mbox{ is countable, possibly empty.}$$
Assume that $T:D\longrightarrow X^*$ is a single-valued operator. If   $T$ is continuous and the restriction $T|_{D\setminus C}$ is locally monotone, then $T$ is monotone on $D$
\end{prop}

\begin{prop}\label{pp} Let $X$ be a real Banach space and let $D\subseteq X$ be open and convex. Let $C\subseteq D$  be a set closed regarding $D,$ with empty interior, satisfying the following condition
$$ \forall\,x,y\in D\setminus C \mbox{ the set } [x,y]\cap C \mbox{ is countable, possibly empty.}$$ Assume that $T:D\To X^*$ is a hemicontinuous operator with the property that $\<T(z),y-x\>\neq 0$ for all $z\in [x,y]\cap C,\,x,y\in D,\, x\neq y.$ Then the following hold.
\begin{itemize}
\item[(a)]\emph{(Theorem  3.6, \cite{La})} If $T|_{D\setminus C}$ is locally strictly quasimonotone, then $T$ is strictly quasimonotone on $D.$
\item[(b)]\emph{(Theorem  3.7, \cite{La})} If $T|_{D\setminus C}$ is locally pseudomonotone, then $T$ is pseudomonotone on $D.$
\item[(c)]\emph{(Theorem  3.8, \cite{La})} If $T|_{D\setminus C}$ is locally strictly pseudomonotone, then $T$ is strictly pseudomonotone on $D.$
\end{itemize}
\end{prop}

Recall that a set $V\subseteq X$ is of first category in the sense of Baire, if $V=\bigcup_{n=1}^{\infty }V_{n}$, where $V_{n}\subseteq X,\,n\in\Natural$ are nowhere dense sets i. e. ${\inte}\left( {\cl}(V_{n})\right) =\emptyset $.


  Note that condition $$ \forall\,x,y\in D\setminus C \mbox{ the set } [x,y]\cap C \mbox{ is countable, possibly empty}$$   imposed on the set $C$ in Proposition \ref{pp}, has been  weakened in \cite{KaPiLa}. More precisely, in \cite{KaPiLa} it  was established the following result.

\begin{prop}\label{ppp}\emph{(Theorem 3, \cite{KaPiLa})}  Let $X$ be a real Banach space and let $V\subseteq X$  be a set of first category with the following property
$$ \forall x,y\in D\setminus V,\mbox{ the set }[x,y]\cap V\mbox{ is countable, possibly empty.}$$
 Assume that  $T:D\To X^*$ is hemicontinuous and locally monotone on $D\setminus V$. Then $T$ is monotone on $D$.
\end{prop}

\begin{rmk}\rm Note that the conclusion of Proposition \ref{ppp} remains valid if the condition $ \forall x,y\in D\setminus V,\mbox{ the set }[x,y]\cap V\mbox{ is countable, possibly empty}$  is replaced by the condition
$$ \forall x,y\in D\setminus V,\mbox{ the set }[x,y]\cap V\mbox{ contains no nonempty perfect subsets.}$$
\end{rmk}

\begin{rmk}\rm\label{r22}
It is worth mentioning that Proposition \ref{ppp} cannot be extended to the set-valued case since, according to a theorem of Kenderov (see \cite{Ke}, Proposition 2.6), a monotone operator that is lower semicontinuous at a point of its domain must be single-valued at that point.

Furthermore, Proposition \ref{ppp} does not remain valid for set-valued upper semicontinuous operators (see \cite{KaPiLa}, Example 2). Next we also provide an example of an upper semicontinuous set-valued operator which is strictly pseudomonotone on a dense subset  of its open and convex domain, but which is  not even quasimonotone on its domain.
\end{rmk}

\begin{ex}\rm
\label{ex2} Consider the operator
\begin{equation*}
F:\mathbb{R} \rightrightarrows\mathbb{R} ,\, F(x):=\left\{
\begin{array}{ll}
x+2,\mbox{ if } x\in{\mathbb{R}}\setminus\{0\}, \\
\, [-1,3],\mbox{ if } x=0.
\end{array}
\right.
\end{equation*}

Let $V:=\{0\}.$ Then, $F$ is upper semicontinuous on $D(F)={\mathbb{R}}$ and
strictly pseudomonotone on $D(F) \setminus V,$ but is not even quasimonotone
on $D(F).$\\
Obviously $D(F)\setminus V={\mathbb{R}}\setminus\{0\}$ is dense in $D(F)={\mathbb{R}}.$ Since $F(x)=x+2$ on $D(F)\setminus V,$ it is obvious that there is strictly pseudomonotone. It can easily be verified that $F$ is upper semicontinuous on ${\mathbb{R}}.$ But, for $(x,x^*)=(-1,1)\in G(F)$
and $(y,y^*)=(0,-1)\in G(F),$ we get $\langle x^*,y-x\rangle=1>0$ and $\langle y^*,y-x\rangle=-1<0,$ which shows that $F$ is not quasimonotone on ${\mathbb{R}}.$

\end{ex}

\begin{rmk}\rm
According to Example \ref{ex1}, Example \ref{ex2} and Remark \ref{r22} we can obtain global generalized monotonicity results for  a set-valued operator, based on its local appropriate generalized monotonicity property, only in the case when the operator is lower semicontinuous and the mentioned monotonicity property is one of the following: strict quasimonotonicity, pseudomonotonicity or strict pseudomonotonicity, respectively.
\end{rmk}

\subsection{Self segment-dense sets}
\noindent

In  \cite{DTL}, Definition 3.4, The Luc has introduced the  notion of a so-called \emph{segment-dense} set. Let $V\subseteq X$ be a convex set. One says that the set $U\subseteq V$ is segment-dense in  $V$ if for each $x\in V$ there can be found $y\in U$  such that $x$ is a cluster point of the set $[x,y]\cap U.$

In what follows we introduce a denseness notion, slightly different from the concept of The Luc presented above, but which is sufficient for the conditions in Proposition \ref{p},  Proposition \ref{pp}, respectively Proposition \ref{ppp}.

\begin{defn}\label{dd} Consider the sets  $U\subseteq V\subseteq X$ and assume that $V$ is convex.

 We say that $U$ is self segment-dense  in $V$ if   $U$ is dense in $V$ and $$\forall x,y\in U,\mbox{  the set }[x,y]\cap U\mbox{  is dense in }[x,y].$$
\end{defn}

\begin{rmk}\rm
Obviously in one dimension the concepts of a segment-dense set  respectively a self segment-dense set are equivalent to the concept of a dense set.

Assume that $C, D,$ respectively $V$ are sets as in Proposition \ref{p}, Proposition \ref{pp}, respectively Proposition \ref{ppp}. Then the set $D\setminus C$, respectively the set $D\setminus V$ is self segment-dense  in $D.$
\end{rmk}

In what follows we provide some examples of self segment-dense sets.

\begin{ex}\rm (see also Example 3.1, \cite{LYAK})\label{EEE}  Let $V$ be the two dimensional Euclidean space $\Real^2$ and define $U$ to be the set
$$U :=\{(p,q) \in\Real^2 : p\in \mathbb{Q},\, q\in\mathbb{Q}\},$$
where $\mathbb{Q}$ denotes the set of all rational numbers. Then, it is clear that $U$ is dense in $\Real^2.$ On the other hand $U$ is not segment-dense in $\Real^2,$ since for $x=(0, \sqrt{2})\in \Real^2 $ and for every $y=(p,q)\in U$, one has $[x, y] \cap U = \{y\}.$

It can easily be observed that  $U$ is also  self segment-dense  in $\Real^2$, since for every $x,y\in U$ $x=(p,q),\,y=(r,s)$ we have $[x,y]\cap U=\{(p+t(r-p),q+t(s-q)): t\in[0,1]\cap\mathbb{Q}\},$ which is obviously dense in $[x,y].$
\end{ex}

\begin{ex}\rm Let $V^{0}\left( \left[ 0,1\right] ,\mathbb{R}\right) $ be the space of
piecewise constant $\mathbb{R}$-valued functions.

Obviously $V^{0}\left( \left[ 0,1\right] ,\mathbb{Q}\right) $ is dense in $%
V^{0}\left( \left[ 0,1\right] ,\mathbb{R}\right) $. So, it follows from the
density of $V^{0}\left( \left[ 0,1\right] ,\mathbb{R}\right) $ in $%
L^{2}\left( \left[ 0,1\right] ,\mathbb{R}\right) $ that
\[
V^{0}\left( \left[ 0,1\right] ,\mathbb{Q}\right) \text{ is dense in }
L^{2}\left( \left[ 0,1\right] ,\mathbb{R}\right) .
\]
Moreover, $V^{0}\left( \left[ 0,1\right] ,\mathbb{Q}\right) $ is self
segment-dense in $L^{2}\left( \left[ 0,1\right] ,\mathbb{R}\right) $, and
for any $u,v\in V^{0}\left( \left[ 0,1\right] ,\mathbb{Q}\right) $ the
segment $\left[ u,v\right] $ is not contained in $V^{0}\left( \left[ 0,1%
\right] ,\mathbb{Q}\right) $. However, $V^{0}\left( \left[ 0,1\right] ,%
\mathbb{Q}\right) $ is not segment-dense (in the sense of The Luc) in $%
L^{2}\left( \left[ 0,1\right] ,\mathbb{R}\right) $, as can be seen by taking
$u\left( x\right) =\sqrt{2}$ for all $x\in \left[ 0,1\right] $ and for any $%
v\in V^{0}\left( \left[ 0,1\right] ,\mathbb{Q}\right) $ we have that $\left[
u,v\right] \cap V^{0}\left( \left[ 0,1\right] ,\mathbb{Q}\right) =\{v\}$.

\bigskip

The proof of the fact that $V^{0}\left( \left[ 0,1\right] ,\mathbb{R}\right)
$ is dense in $L^{2}\left( \left[ 0,1\right] ,\mathbb{R}\right) $ relies on
the following classical arguments.

Let $\varepsilon >0$ be given. We have that $C_{c}^{\infty }\left( \left[ 0,1%
\right] ,\mathbb{R}\right) $ is dense in $L^{2}\left( \left[ 0,1\right] ,%
\mathbb{R}\right) $ (see for example \cite{B}), so for a given $u\in
L^{2}\left( \left[ 0,1\right] ,\mathbb{R}\right) $ the exists $%
u^{\varepsilon }\in C_{c}^{\infty }\left( \left[ 0,1\right] ,\mathbb{R}%
\right) $ such that%
\[
\left\Vert u-u^{\varepsilon }\right\Vert _{2}<\varepsilon /2.
\]%
Now since $u^{\varepsilon }$ is continuous on the compact set $[0,1]$,
uniform continuity guarantees that there exists $\delta >0$ such that for
any $x,y\in \left[ 0,1\right] $ with $\left\vert x-y\right\vert <\delta $ we
have%
\[
\left\vert u^{\varepsilon }\left( x\right) -u^{\varepsilon }\left( y\right)
\right\vert <\varepsilon /2.
\]%
If we take a regular grid (partition) with step $h$ small enough ($h<\delta $%
) and choose $u^{h}\in V^{0}\left( \left[ 0,1\right] ,\mathbb{R}\right) $
such that $u^{h}\left( \left( i-\frac{1}{2}\right) h\right) =u^{\varepsilon
}\left( \left( i-\frac{1}{2}\right) h\right) $ for $i:=0,1,\ldots \frac{1}{h}
$ then%
\[
\left\Vert u^{h}-u^{\varepsilon }\right\Vert _{\infty }<\varepsilon /2.
\]%
Hence%
\begin{eqnarray*}
\left\Vert u-u^{h}\right\Vert _{2} &\leq &\left\Vert u-u^{\varepsilon
}\right\Vert _{2}+\left\Vert u^{\varepsilon }-u^{h}\right\Vert _{2} \\
&\leq &\varepsilon /2+\varepsilon /2.
\end{eqnarray*}

\end{ex}

In the next section we show that  a lower semicontinuous set-valued operator that possesses  a local generalized monotonicity property on a  self segment-dense subset of its domain is actually globally generalized monotone. According to the previous remark, this result not only extends but also improves the results presented in Proposition \ref{pp}.  


\section{On the local generalized monotonicity of set-valued maps on dense subsets}

In this section we show that the local strict quasimonotonicity, local pseudomonotonicity, respectively local strict pseudomonotonicity property of a lower semicontinuous set-valued operator  on a self segment-dense subset of its domain provides the global appropriate monotonicity of that operator under some mild conditions. We attain this goal gradually, by showing first that the previous statement holds for real set-valued lower semicontinuous functions of one real variable. We also show that our conditions cannot be weakened. As consequences we obtain some results that improve the results stated in \cite{La}.

\subsection{Locally quasiomonotone and locally strictly quasimonotone operators}
\noindent

In this paragraph we study  the local quasimonotonicity, respectively the local strict quasimonotonicity of set-valued operators on dense sets.  As it was expected some additional conditions are needed in order to assure that local generalized monotonicity on a dense subset of a lower semicontinuous set-valued operator implies its global counterpart. Such conditions will be given and by an example we show that our conditions are essential in obtaining these results.

Let $F:I\subseteq{\mathbb{R}}\rightrightarrows {\mathbb{R}}$ be a set-valued function. Let us introduce the following notations.
\begin{equation*}
F(x)\ge 0\Leftrightarrow x^*\ge 0,\,\forall x^*\in F(x),\mbox{ and } F(x)>0, \mbox{ if } F(x)\ge0\mbox{ and }0\not\in F(x).
\end{equation*}
Obviously the inequalities $F(x)\le 0$, respectively $F(x)<0$ can be introduced analogously.

\begin{lem}
\label{l31} Let $I\subseteq {\mathbb{R}}$ be an interval and consider the lower semicontinuous set-valued function $F:I\rightrightarrows {\mathbb{R}}.$
Let $J\subseteq I$ dense in $I$ and assume that  $F$  is locally quasimonotone on $J.$ Then, for any fixed $x\in I$ one has $F(x)\ge0,$ or $F(x)\le 0.$
\end{lem}

\begin{proof}
Let $x\in I$, and suppose, that there exist $x_1^*,x_2^*\in F(x),\, x_1^*>0,\,x_2^*<0.$ Consider the intervals $V_1=(x_1^*-\epsilon,x_1^*+ \epsilon),$ respectively $V_2=(x_2^*-\epsilon,x_2^*+\epsilon),$ where $\epsilon>0$ such that $x_1^*-\epsilon>0$ and $x_2^*+\epsilon<0.$ By the lower semicontinuity of $F$ there exists an open neighborhood $U_x$ of $x$, such that for every $y\in U_x\cap I$ one has $F(y)\cap V_1\neq\emptyset$ and $F(y)\cap V_2\neq\emptyset.$ Let $y\in U_x\cap J.$ Then, according to the hypothesis of the lemma, there exists an open neighborhood $U_y$ of $y$ such that $F|_{U_y\cap J}$ is quasimonotone.

Let $U=U_x\cap U_y\cap J$ and let $u,v\in U,\, u\neq v.$ Then, there exist $u_1^*,u_2^*\in F(u),\, v_1^*,v_2^*\in F(v)$ such that $u_1^*>0,v_1^*>0,\mbox{ and }u_2^*<0,v_2^*<0.$

Hence,
\begin{equation*}
\<u_1^*,v-u\>>0,\mbox{ or } \<u_2^*,v-u\>>0,
\end{equation*}
on the other hand
\begin{equation*}
\<v_1^*,v-u\><0,\mbox{ or } \<v_2^*,v-u\><0,
\end{equation*}
which contradicts the fact that $F$ is quasimonotone in $U.$

Consequently, for any $x\in I$ one has $F(x)\ge 0,$ or $F(x)\le 0.$
\end{proof}

Further, we need the following result which proof is straightforward.

\begin{prop}
\label{p32} Let $F:I\subseteq{\mathbb{R}}\rightrightarrows {\mathbb{R}}$ a set-valued function.

\begin{itemize}
\item[(a)] If $F(x)\ge 0$ for every $x\in I$, (respectively $F(x)\le0$ for every $x\in I$) , then $F$ is quasimonotone on $I$.

\item[(b)] If $F(x)=\{0\}$ for some $x\in I$, and $F(z)\le 0$ for every $z<x,\,z\in I,$ respectively $F(y)\ge 0$ for every $x< y,\,y\in I,$ then $F$ is
quasimonotone.
\end{itemize}
\end{prop}

Now we are able to state and prove the following result concerning on the strict quasimonotonicity of a set-valued function.
\begin{thm}
\label{t31} Let $I\subseteq {\mathbb{R}}$ be an interval and consider the lower semicontinuous set-valued function $F:I\rightrightarrows {\mathbb{R}}.$
Let $J\subseteq I$ dense in $I$ and assume that $F$ is locally {strictly quasimonotone} on $J.$ Further, assume that $0\not\in F(x),$
for all $x\in I\setminus J.$ Then $F$ is {strictly quasimonotone} on $I.$
\end{thm}

\begin{proof}
According to Lemma \ref{l31}, for any $x\in I$ one has $F(x)\ge 0,$ or $F(x)\le 0.$ Note that according to the hypothesis of the theorem, for any $x\in I\setminus J$ one has $F(x)>0,$ or $F(x)<0.$

We show next that $F$ is locally strictly quasimonotone on $I$, that is, every $x\in I$ admits an open neighborhood $U$, such that $F|_{U\cap I}$ is strictly quasimonotone. Let $x\in I$ and assume that $F(x)\neq\{0\}.$ Then we have $F(x)\ge 0,$ or $F(x)\le 0.$ Assume that $F(x)\ge 0$ the other case can be treated
analogously. Obviously there exists $x^*\in F(x), x^*>0.$ Let $V=(x^*-\epsilon,x_1^*+\epsilon),$ an open interval, where $\epsilon>0$ such
that $x^*-\epsilon>0.$ Since $F$ is lower semicontinuous, there exists an open neighborhood $U_x$ of $x$, such that for every $y\in U_x\cap I$ we have
$F(y)\cap V\neq\emptyset.$ But then, according to Lemma \ref{l31} $F(y)\ge 0,$ for all $y\in U_x\cap I,$ hence from Proposition \ref{p32} we get that $F$
is quasimonotone on $U_x\cap I.$ Let $u,v\in U_x\cap I,\, u\neq v.$ Then, by the denseness of $J$ in $I$, we obtain that there exists $w\in(u,v)\cap J,$ and
according to the hypothesis of the theorem $w$ admits an open neighborhood $U_w$ such that $F|_{U_w\cap J}$ is strictly quasimonotone. This in particular shows, that
there exists $z\in(u,v)$ and $z^*\in F(z)$, such that $\<z^*,v-u\>\neq 0.$ Hence, $F$ is strictly quasimonotone on $U_x\cap I.$

If $F(x)=\{0\}$, then $x\in J$, hence there exists an open neighborhood $U_x$ of $x$ such that $F|_{U_x\cap J}$ is strictly quasimonotone. Note that according to
Proposition \ref{p32}, in case that there exists an open neighborhood $U$ of $x$ such that one of the following conditions is fulfilled

\begin{itemize}
\item[(a)] $F(y)\ge0$ for every $y\in U\cap I,$ or,

\item[(b)] $F(y)\le 0$ for every $y\in U\cap I,$ or,

\item[(c)] $F(z)\le 0$ and $F(y)\ge 0$ or every $z\le x\le y,\,y,z\in U\cap
I,$
\end{itemize}

then $F|_{U\cap I}$ is strictly quasimonotone.

Assume now, that none of the conditions $(a)-(c)$ is fulfilled. Then, for every neighborhood $U$ of $x$, there exists $y,z\in U\cap I,\, z>x,\,y>x$ (or $z<x,\,y<x,$) such that $F(y)\ge 0,\, F(y)\neq\{0\}$ and $F(z)\le 0,\,F(z)\neq\{0\}.$ Let $y,z\in U_x\cap I,\, z>x,\,y>x,$ (the case $y,z\in U\cap I,\, z<x,\,y<x$ can be treated similarly). Obviously, we can assume that $z>y.$ According to the previous part of the proof there exist an open neighborhood $U_y$ of $y,$ respectively $U_z$ of $z$ such that $F(u)\ge 0$ for all $u\in{U_y\cap I,}$ respectively $F(v)\le0$ for all $v\in{U_z\cap I}.$

Let $u\in U_y\cap U_x\cap J,$ respectively $v\in U_z\cap U_x\cap J,$ such that $u<v,$ and with the property that there exist $u^*\in F(u),\, u^*>0$,
respectively $v^*\in F(v),\,v^*<0.$ Such $u^*$ and $v^*$ must exist since $F$ is strictly quasimonotone on $U_x\cap J.$ Then, $\<u^*,v-u\>>0$ and $\<v^*,v-u\><0,$
contradiction.

We have shown, that $F$ is locally strictly quasimonotone on $I$. According to Proposition \ref{p31}, $F$ is globally strictly quasimonotone on $I$ and the proof is completely done.
\end{proof}

In what follows we extend Theorem \ref{t31} for  lower semicontinuous operators that are locally strictly quasimonotone on a self segment-dense  subset of their convex domain.

\begin{thm}\label{t32} Let $X$ be a real Banach space and let $X^*$ be its topological dual. Let $T:X\rightrightarrows X^*$ be a lower semicontinuous operator with
convex domain $D(T).$ Let $D\subseteq D(T)$ self segment-dense in $D(T)$ and assume that $T$ is locally strictly quasimonotone on $D.$ Assume further, that the following condition holds:
\begin{equation}  \label{e1}
\forall x,y\in D(T),\, z\in[x,y]\cap(D(T)\setminus D)\mbox{ and }z^*\in T(z) \mbox{ one has }\<z^*,y-x\>\neq0.
\end{equation}
Then $T$ is {strictly quasimonotone}.
\end{thm}

\begin{proof}
 First of all observe, that for every $x,y\in D(T),\,x\neq y$ there exists $z\in(x,y)$ and $z^*\in T(z)$ such that $\<z^*,y-x\>\neq0.$ Indeed, if $[x,y]\cap D$ is dense in $[x,y]$ then the statement follows from the local strict quasimonotonicity property of $T$ on $D.$ Otherwise, there exists $z\in(x,y),\,z\not\in D$. But then the statement follows from condition (\ref{e1}). Hence, it is enough to show, that for all $(x,x^*),(y,y^*)\in G(T),$ such that $\<x^*,y-x\>>0$ we have $\<y^*,y-x\>\ge0.$

We divide the proof into two cases.

If $x,y\in D$ then consider the real set-valued function
\begin{equation*}
F_{x,y}:[0,1]\rightrightarrows{\mathbb{R}},\,F_{x,y}(t)=\{\<z^*,y-x\>:z^*\in T(x+t(y-x))\}.
\end{equation*}

Obviously $T$ is lower semicontinuous on $[x,y].$ Since $T$ is locally strict quasimonotone on $D$, obviously $T$ is locally strict quasimonotone on $[x,y]\cap D.$ On the other hand $D$ is self segment-dense in $D(T)$, hence $[x,y]\cap D$ is dense in $[x,y].$ Let
\begin{equation*}
J=\{t\in[0,1]:x+t(y-x)\in [x,y]\cap D\}.
\end{equation*}
Obviously $J$ is dense in $[0,1]$ and $F_{x,y}$ is locally strict quasimonotone on $J$. On the other hand, from \ref{e1} we obtain, that $0\not\in F_{x,y}(s)$ for
every $s\in[0,1]\setminus J.$ Hence, according to Theorem \ref{t31}, $F_{x,y} $ is strictly quasimonotone on $[0,1].$ In particular, the latter shows, that for $t^*\in
F_{x,y}(0)$ one has $t^*>0$ then $s^*\ge0$ for all $s^*\in F_{x,y}(1)$, or equivalently if $\<x^*,y-x\>>0$ for some $x^*\in T(x)$ then $\<y^*,y-x\>\ge0$ for all $y^*\in T(y).$

Let now $x,y\in D(T)$ arbitrary and assume that $\<x^*,y-x\>>0.$ We show that $\<y^*,y-x\>\ge0,$ for all $y^*\in T(y).$ Since $T$ is lower
semicontinuous, according to Remark \ref{r21}, for every sequence $({x_n})\subseteq D(T)$ such that $x_n\longrightarrow x$ there exists a sequence of
elements $x^*_n \in T(x_n)$ such that $x^*_{n}\longrightarrow x^*,$ respectively for every sequence $({y_n})\subseteq D(T)$ such that $y_n\longrightarrow y$ and for every $y^*\in T(y)$ there exists a sequence of elements $y^*_n \in T(y_n)$ such that $y^*_{n}\longrightarrow y^*.$ Let us fix $y^*\in T(y).$ Since $D$ is dense in $D(T)$, one can consider $(x_n),(y_n)\subseteq D.$ Obviously $\<x_n^*,y_n-x_n\>\longrightarrow \<x^*,y-x\>>0,$ hence $\<x_n^*,y_n-x_n\>>0$ if $n$ is big enough. According to the first part of the proof $\<y_n^*,y_n-x_n\>\ge0.$ By taking the limit $n\longrightarrow\infty$ one obtains $\<y^*,y-x\>\ge 0.$
\end{proof}

According to the next corollary the previous result improves Theorem  3.6 from \cite{La}.

\begin{cor} Let $T:D(T)\To X^*$ be a continuous operator with convex domain $D(T).$ Let $D\subseteq D(T)$ self segment-dense in $D(T)$ and assume that $T$ is locally strictly quasimonotone on $D.$ Assume further, that the following condition holds:
\begin{equation*}
\forall x,y\in D(T),\, z\in[x,y]\cap(D(T)\setminus D) \mbox{ one has }\<T(z),y-x\>\neq0.
\end{equation*}
Then $T$ is strictly quasimonotone.
\end{cor}

\subsection{Locally pseudomonotone and locally strictly pseudomonotone operators}
\noindent

In what follows we study  the local pseudomonotonicity, respectively the local strict pseudomonotonicity of set-valued operators on self segment-dense sets. We need the following result which proof is straightforward.

\begin{prop}
\label{p33} Let $F:I\subseteq{\mathbb{R}}\rightrightarrows {\mathbb{R}}$ a set-valued function.

\begin{itemize}
\item[(a)] If $F(x)> 0$, for every $x\in I$, (respectively $F(x)<0$ for every $x\in I$), then $F$ is strictly pseudomonotone on $I$.

\item[(b)] If $F(x)=\{0\}$ for some $x\in I$, and $F(z)< 0$ for every $z<x,\,z\in I,$ respectively $F(y)> 0$ for every $x< y,\,y\in I,$ then $F$ is
strictly pseudomonotone on $I$.

\item[(c)] If $I$ is closed from left, $x$ is the left endpoint of $I$ such that  $F(x)\ge 0$ and $F(y)>0$ for all $y>x,\, y\in I$, then $F$ is strictly pseudomonotone on $I$.
\item[(d)] If $I$ is closed from right, $x$ is the right endpoint of $I$ such that $F(x)\le 0$ and $F(y)<0$ for all $y<x,\, y\in I$, then $F$ is strictly pseudomonotone on $I$.
\end{itemize}
\end{prop}
The following result ensures the pseudomonotonicity of a locally pseudomonotone set-valued function on a dense subset.
\begin{thm}\label{t33} Let $I\subseteq {\mathbb{R}}$ be an interval and consider the lower semicontinuous set-valued function $F:I\rightrightarrows {\mathbb{R}}.$
Let $J\subseteq I$ dense in $I$ and assume that $F$  is locally {pseudomonotone} on $J.$ Further, assume that $0\not\in F(x),$ for all $x\in I\setminus J.$ Then $F$ is {pseudomonotone} on $I.$
\end{thm}

\begin{proof}
We show that $F$ is locally pseudomonotone on $I$. According  to Lemma \ref{l31}, for any $x\in I$ one has $F(x)\ge 0,$ or $F(x)\le 0.$
Note that according to the hypothesis of the theorem, for any $x\in I\setminus J$ one has $F(x)>0,$ or $F(x)<0.$

 Let $x\in I$ and assume that $F(x)\neq\{0\}.$ We have $F(x)\ge 0,$ or $F(x)\le 0.$ Assume that $F(x)\ge 0$ the other case can be treated analogously. Obviously there exists $x^*\in F(x), x^*>0.$ Let $V=(x^*-\epsilon,x_1^*+\epsilon),$ an open interval, where $\epsilon>0$ such
that $x^*-\epsilon>0.$ Since $F$ is lower semicontinuous, there exists an open neighborhood $U_x$ of $x$, such that for every $y\in U_x\cap I$ we have
$F(y)\cap V\neq\emptyset.$ But then, according to Lemma \ref{l31} $F(y)\ge 0,$ for all $y\in U_x\cap I.$ One can observe as well that $F(y)\neq \{0\},$ for all $y\in U_x\cap I.$

We show next, that $F(y)>0,$ for all $y\in U_x\cap I,\,y\neq x.$ Indeed, if $0\in F(y)$ for some $y\in U_x\cap I,$ then according to the hypothesis we have $y\in J.$ Hence, there exists an open neighborhood $U_y\subseteq U_x\cap I$ of $y$ such that $F|_{U_y\cap J}$ is pseudomonotone. But then for $y_1<y,\,y_1\in U_y\cap J$ we have an $y_1^*\in F(y_1),\, y_1^*>0$, hence $\<0,y_1-y\>=0$ and $\<y_1^*,y_1-y\><0,$ contradiction.

Note that we can have $0\in F(x)$ if, and only if, $I$ is closed from left and $x$ is the left endpoint of $I.$ Otherwise, according to the previous part of the proof for $x_1<x,\,x_1\in U_x\cap J$ we have an $x_1^*\in F(x_1),\, x_1^*>0$, hence $\<0,x_1-x\>=0$ and $\<x_1^*,x_1-x\><0,$ contradiction.

Hence, according to Proposition \ref{p33}, in this case $F$ is strictly pseudomonotone on $U_x\cap I.$

If $F(x)=\{0\}$, then $x\in J$, hence there exists an open neighborhood $U_x$ of $x$ such that $F|_{U_x\cap J}$ is pseudomonotone.
We show that $F$ is pseudomonotone on $U_x.$ Suppose the contrary, that is, there exist $u,v\in U_x$ and $u^*\in T(u),\,v^*\in T(v)$ such that $\<u^*,v-u\>\ge 0$ and $\<v^*,v-u\><0.$

Assume that $v>u$, the other case can be treated analogously.

We have $u^*\ge 0$, respectively $v^*<0$, hence according to Lemma \ref{l31} $F(v)\le 0.$ Hence, by the lower semicontinuity of $F$ and by the denseness of $J$ in $I$, there exists $w\in(u,v),\,w\in J$ such that $w^*< 0$ for some $w^*\in F(w).$

If $u\in J$ then $\<u^*,w-u\>\ge 0$ and $\<w^*,w-u\><0,$  which contradicts the fact that $F$ is locally pseudomonotone on $U_x\cap J.$

If $u\in I\setminus J$ then from $\<u^*,v-u\>\ge 0$ we have $u^*>0$, hence, by the lower semicontinuity of $F$ and by the denseness of $J$ in $I$, there exists $z\in(u,w),\,z\in J$ such that $z^*> 0$ for some $z^*\in F(z).$ But then  $\<z^*,w-z\>\ge 0$ and $\<w^*,w-z\><0,$  contradicts the fact that $F$ is locally pseudomonotone on $U_x\cap J.$

We have shown, that $F$ is locally pseudomonotone on $I$. According to Proposition \ref{p31}, $F$ is globally pseudomonotone on $I$ and the proof is completely done.
\end{proof}

Now we are ready to state and prove one of the main results of this section concerning on pseudomonotonicity of a lower semicontinuous set-valued operator, locally pseudomonotone on a self segment-dense subset of its domain.

\begin{thm}\label{t34} Let $X$ be a real Banach space and let $X^*$ be its topological dual. Let $T:X\rightrightarrows X^*$ be a lower semicontinuous operator with
convex domain $D(T).$ Let $D\subseteq D(T)$ be self segment-dense in $D(T)$ and assume that $T$ is locally {pseudomonotone} on $D.$ Assume further, that the following condition holds:
\begin{equation}  \label{e2}
\forall x,y\in D(T),\, z\in[x,y]\cap(D(T)\setminus D)\mbox{ and }z^*\in T(z) \mbox{ one has }\<z^*,y-x\>\neq0.
\end{equation}
Then $T$ is  {pseudomonotone}.
\end{thm}

\begin{proof}
For $x,y\in D$ consider the real set-valued function
\begin{equation*}
F_{x,y}:[0,1]\rightrightarrows{\mathbb{R}},\,F_{x,y}(t)=\{\<z^*,y-x\>:z^*\in
T(x+t(y-x))\}.
\end{equation*}

Obviously $T$ is lower semicontinuous on $[x,y].$ Since $T$ is locally pseudomonotone on $D$, obviously $T$ is locally pseudomonotone on $[x,y]\cap D.$ On the other hand  $[x,y]\cap D$ is dense in $[x,y].$ Let
\begin{equation*}
J=\{t\in[0,1]:x+t(y-x)\in [x,y]\cap D\}.
\end{equation*}
Obviously $J$ is dense in $[0,1]$ and $F_{x,y}$ is locally pseudomonotone on $J$. On the other hand, from (\ref{e2}) we obtain, that $0\not\in F_{x,y}(s)$ for
every $s\in[0,1]\setminus J.$ Hence, according to Theorem \ref{t32}, $F_{x,y} $ is pseudomonotone on $[0,1].$ In particular, the latter shows, that for $t^*\in
F_{x,y}(0)$ one has $t^*\ge0$ then $s^*\ge0$ for all $s^*\in F_{x,y}(1)$, or equivalently if $\<x^*,y-x\>\ge0$ for some $x^*\in T(x)$ then $\<y^*,y-x\>\ge0$ for all $y^*\in T(y).$

Let $x\in D(T),\,y\in D(T)\setminus D$  and assume that $\<x^*,y-x\>>0.$ We show that $\<y^*,y-x\>>0,$ for all $y^*\in T(y).$ Since $T$ is lower
semicontinuous, according to Remark \ref{r21}, for every sequence $({x_n})\subseteq D(T)$ such that $x_n\longrightarrow x$ there exists a sequence of
elements $x^*_n \in T(x_n)$ such that $x^*_{n}\longrightarrow x^*,$ respectively for every sequence $({y_n})\subseteq D(T)$ such that $y_n\longrightarrow y$ and for every $y^*\in T(y)$ there exists a sequence of elements $y^*_n \in T(y_n)$ such that $y^*_{n}\longrightarrow y^*.$ Let us fix $y^*\in T(y).$ Since $D$ is dense in $D(T)$, one can consider $(x_n),(y_n)\subseteq D.$ Obviously $\<x_n^*,y_n-x_n\>\longrightarrow \<x^*,y-x\>>0,$ hence $\<x_n^*,y_n-x_n\>>0$ if $n$ is big enough. According to the first part of the proof $\<y_n^*,y_n-x_n\>>0.$ By taking the limit $n\longrightarrow\infty$ one obtains $\<y^*,y-x\>\ge 0.$ Since $y\in D(T)\setminus D$ then by the assumption of the theorem  $\<y^*,y-x\>\neq 0,$ hence $\<y^*,y-x\>> 0.$
\end{proof}

In single-valued case we have the following  corollary which improves Theorem  3.7 from \cite{La}.

\begin{cor} Let $T:D(T)\To X^*$ be a continuous operator with convex domain $D(T).$ Let $D\subseteq D(T)$ self segment-dense in $D(T)$ and assume that $T$ is locally pseudomonotone on $D.$ Assume further, that the following condition holds:
\begin{equation*}
\forall x,y\in D(T),\, z\in[x,y]\cap(D(T)\setminus D) \mbox{ one has }\<T(z),y-x\>\neq0.
\end{equation*}
Then $T$ is pseudomonotone.
\end{cor}

We have some similar result to those established in Theorem \ref{t33} for locally strict pseudomonotone set-valued functions.

\begin{thm}\label{t35} Let $I\subseteq {\mathbb{R}}$ be an interval and consider the lower semicontinuous set-valued function $F:I\rightrightarrows {\mathbb{R}}.$
Let $J\subseteq I$ dense in $I$ and assume that  $F$  is  locally {strictly pseudomonotone} on $J.$
Further, assume that $0\not\in F(x),$ for all $x\in I\setminus J.$ Then $F$ is {strictly pseudomonotone} on $I.$
\end{thm}
\begin{proof} We prove that $F$ is locally strictly pseudomonotone on $I.$ Let $x\in I$ and $U_x$ a neighborhood of $x$, such that $F$ is locally strict pseudomonotone on $U_x\cap J.$ Then, according to Theorem \ref{t32} $F$ is pseudomonotone on $U_x\cap I.$ Assume that there exist $u,v\in U_x\cap I$ and $u^*\in F(u),\, v^*\in F(v)$ such that $\<u^*,v-u\>\ge 0$ and $\<v^*,v-u\>\le 0.$ Since $F$ is pseudomonotone on $U_x\cap I$, we must have $\<v^*,v-u\>=0.$ Hence $v^*=0$ and due to the hypothesis of the theorem we obtain $v\in J.$ But $\<v^*,u-v\>=0$ and by the pseudomonotonicity of $F$ on $U_x\cap I$ we get $\<u^*,u-v\>\ge 0.$ Hence $u^*=0$ and arguing as before $u\in J.$

But $F$ is strictly pseudomonotone on  $U_x\cap J$ hence
$$\<u^*,v-u\>= 0\Rightarrow  \<v^*,v-u\>> 0,$$
contradiction.

Since $F$ is locally strict pseudomonotone on $I$ according to Proposition \ref{p31} $F$ is strict pseudomonotone on $I.$
\end{proof}

In infinite dimension we have the following result concerning on strict pseudomonotonicity of a lower semicontinuous set-valued operator, locally strictly pseudomonotone on a self segment-dense subset of its domain..

\begin{thm} \label{t36} Let $X$ be a real Banach space and let $X^*$ be its topological dual. Let $T:X\rightrightarrows X^*$ be a lower semicontinuous operator with
convex domain $D(T).$ Let $D\subseteq D(T)$ be self segment-dense in $D(T)$  and assume that $T$ is locally {strictly pseudomonotone} on $D.$ Assume further, that the following condition holds:
\begin{equation}  \label{e3}
\forall x,y\in D(T),\, z\in[x,y]\cap(D(T)\setminus D)\mbox{ and }z^*\in T(z) \mbox{ one has }\<z^*,y-x\>\neq0.
\end{equation}
Then $T$ is {strictly pseudomonotone}.
\end{thm}

\begin{proof}
For $x,y\in D$ the proof is similar to the proof of Theorem \ref{t34}. If $x\in D(T),\,y\in D(T)\setminus D$ then by (\ref{e3}) one has  $\<y^*,y-x\>\neq 0,$ for all $ y^*\in T(y)$, hence the implication
$$\<x^*,y-x\>\ge0\Rightarrow \<y^*,y-x\>>0$$
can be obtained as in the proof of Theorem \ref{t34}.
\end{proof}

In single-valued case we have the following  corollary which improves Theorem  3.8 from \cite{La}.

\begin{cor} Let $T:D(T)\To X^*$ be a continuous operator with convex domain $D(T).$ Let $D\subseteq D(T)$   self segment-dense in $D(T)$ and assume that $T$ is locally strictly pseudomonotone on $D.$ Assume further, that the following condition holds:
\begin{equation*}
\forall x,y\in D(T),\, z\in[x,y]\cap(D(T)\setminus D) \mbox{ one has }\<T(z),y-x\>\neq0.
\end{equation*}
Then $T$ is strictly pseudomonotone.
\end{cor}

\begin{rmk}\rm
The condition
$$\forall x,y\in D(T),\, z\in[x,y]\cap(D(T)\setminus D)\mbox{ and }z^*\in T(z) \mbox{ one has }\<z^*,y-x\>\neq0$$ in the hypothesis of Theorem \ref{t32}, Theorem \ref{t34}, respectively Theorem \ref{t36}, in particular the condition $0\not\in F(x),$ for all $x\in I\setminus J$ in the hypothesis of Theorem \ref{t31}, Theorem \ref{t33}, respectively Theorem \ref{t35}, is essential as the next example shows.
\end{rmk}

\begin{ex}\rm
Let
\begin{equation*}
F:\mathbb{R} \rightrightarrows\mathbb{R} ,\, F(x):=\left\{
\begin{array}{lll}
(0,-x],\mbox{ if } x<0, &  &  \\
\, 0,\mbox{ if } x=0, &  &  \\
\, [-x,0),\mbox{ if } x>0.
\end{array}
\right.
\end{equation*}
Then $F$ is lower semicontinuous on ${\mathbb{R}}$ and locally {strictly pseudomonotone} on ${\mathbb{R}}\setminus\{0\}$, but $F$ is not even {quasimonotone} on ${\mathbb{R}}$.
\end{ex}

It is easy to check that $F$ is locally {strictly pseudomonotone} on $D={\mathbb{R}} \setminus\{0\}$ and lower semicontinuous on $D(F)={\mathbb{R}}.$ Obviously ${
\mathbb{R}}\setminus\{0\}$ is self segment-dense in ${\mathbb{R}},$ hence all the assumptions in the hypothesis of Theorem \ref{t32}, Theorem \ref{t34}, respectively  Theorem \ref{t36} are satisfied excepting the one that
$$\forall x,y\in D(F),\, z\in[x,y]\cap(D(F)\setminus D)\mbox{ and }z^*\in F(z) \mbox{ one has }\<z^*,y-x\>\neq0.$$
 Consequently, their conclusion fail.

Indeed, for $(x,x^*)=(-1,1)\in G(F)$ and $(y,y^*)=(1,-1)\in G(F)$ we have $\<x^*,y-x\>=2>0,$ and $\<y^*,y-x\>=-2<0,$ which shows that $F$ is not {quasimonotone}
on ${\mathbb{R}}.$

\section{Generalized convex functions on dense subsets}

In this section we apply the results obtained in Section 3 to prove the generalized convexity of a locally generalized convex function on  a self segment-dense subset of its domain. In the sequel $X$ denotes a real Banach space and $X^*$ denotes its topological dual.

In order to continue our analysis we need the following concepts (see, for instance \cite{C,R}). Let $f:X\longrightarrow\overline{{\mathbb{R}}}={\mathbb{R}}\cup\{-\infty,+\infty\}$ be a given function, and let $x\in X$ such that $f(x)\in{\mathbb{R}}.$ Recall that the Clarke-Rockafellar generalized derivative of $f$ at $x$ in direction $v$ is defined by
\begin{equation*}
f^{\uparrow}(x,v)=\sup_{\epsilon>0}\limsup_{(y,\alpha)\downarrow_f x\, t\downarrow0}\inf_{u\in B(v,\epsilon)}\frac{f(y+tu)-\alpha}{t},
\end{equation*}
where $(y,\alpha)\downarrow_f x$ means $y\longrightarrow x,\,\alpha\longrightarrow f(x),\,\alpha\ge f(y)$ and $B(v,\e)$ denotes the open ball with center $v$ and radius $\e.$

If $f$ is lower semicontinuous at $x$  the Clarke-Rockafellar generalized derivative of $f$ at $x$ in direction $v$   (see \cite{R}) can be expressed as
\begin{equation*}
f^{\uparrow}(x,v)=\sup_{\epsilon>0}\limsup_{y\downarrow_f x\, t\downarrow0}\inf_{u\in B(v,\epsilon)}\frac{f(y+tu)-f(y)}{t},
\end{equation*}
where $y\downarrow_f x$ means $y\longrightarrow x,\,f(y)\longrightarrow f(x).$

When $f$ is locally Lipschitz $f^{\uparrow}$ coincides with the Clarke directional derivative (see \cite{C}), i.e.
\begin{equation*}
f^{C}(x,v)=\limsup_{\overset{y\to x\, u\to v}{t\downarrow0}}\frac{ f(y+tu)-f(y)}{t}.
\end{equation*}

The Clarke-Rockafellar subdifferential  of $f$ at $x$  is given by
\begin{equation*}
\partial^\uparrow  f(x)=\{x^*\in X^*: \<x^*,v\>\le f^\uparrow (x,v),\,\forall v\in X\}.
\end{equation*}

In what follows we present some convexity notions of a real valued function $f:X\longrightarrow{{\mathbb{R}\cup\{\infty\}}}$ (see, for instance \cite{A,ACL,HKS,P1,PQ}). Recall that the domain of $f$ is the set $\dom f=\{x\in X: f(x)\in \Real\}.$ In the sequel we assume that $\dom f$ is a convex subset of $X.$  Recall that the function $f$ is called:

\begin{enumerate}
\item[(1)] \emph{\ convex}, if for all $x,y\in X,\,t\in [0,1]$ one has the following
\begin{equation*}
f(x+t(y-x))\le f(x)+t(f(y)-f(x))\},
\end{equation*}

\item[(2)] \emph{pseudoconvex}, if for all $x,y\in \dom f,$ the following implication holds
\begin{equation*}
f(y)<f(x)\Rightarrow  \forall x^*\in \partial^\uparrow  f (x):\<x^*,y-x\>< 0,
\end{equation*}

\item[(3)] \emph{strictly pseudoconvex}, if for all $x,y\in \dom f,$ the following implication holds
\begin{equation*}
f(y)\le f(x)\Rightarrow  \forall x^*\in \partial^\uparrow  f (x):\<x^*,y-x\><0,
\end{equation*}

\item[(4)] \emph{quasiconvex}, if for all $x,y\in \dom f,\,t\in[0,1]$ one has the following
\begin{equation*}
f(x+t(y-x))\le \max\{f(x),f(y)\},
\end{equation*}

\item[(5)] \emph{strictly quasiconvex}, if for all $x,y\in \dom f,\,x\neq y,\,t\in(0,1)$ one has the following
\begin{equation*}
f(x+t(y-x))< \max\{f(x),f(y)\}.
\end{equation*}
\end{enumerate}

The study of connection between the (generalized) convexity property of a real valued function and appropriate monotonicity of its Clarke-Rockafellar subdifferential ha a rich literature (see for instance \cite{A,ACL,DH,DH1,Ko,P,P1} and the references therein).
In what follows we assume that $f:X\longrightarrow{\mathbb{R}}\cup\{\infty\}$ is locally Lipschitz.
The relation among the previously presented convexity concepts is shown bellow. 

\begin{center}
\begin{tabular}{cc}
  $$ & $f \mbox{ is convex} $\\
  $$ & $\Downarrow $\\
  $ f \mbox{ is strictly pseudoconvex}\Longrightarrow$ & $f \mbox{ is pseudoconvex}$\\
  $\Downarrow$ & $\Downarrow$\\
  $f \mbox{  is strictly quasiconvex}\Longrightarrow$ & $f \mbox{ is quasiconvex}.$\\
\end{tabular}
\end{center}

\begin{rmk}\rm\label{r4} Observe that de definition of pseudoconvexity, respectively strict pseudoconvexity is equivalent to
 the following: for all $x,y\in \dom f,$
\begin{equation*}
\exists x^*\in \partial^\uparrow  f (x):\<x^*,y-x\>\ge0\Rightarrow  f(y)\ge f(x),
\end{equation*}
respectively
\begin{equation*}
\exists x^*\in \partial^\uparrow  f (x):\<x^*,y-x\>\ge0\Rightarrow  f(y)> f(x),
\end{equation*}
\end{rmk}

The following statement  relates the quasiconvexity of a function to the quasimonotonicity of its Clarke-Rockafellar subdifferential (see \cite{ACL,P,P1}):

\begin{prop}[Theorem 4.1 \cite{ACL}]\label{p41} A lower semicontinuous function $f:X\longrightarrow{\mathbb{R}}\cup\{\infty\}$ is quasiconvex, if and only if, $\partial^\uparrow  f$ is quasimonotone.
\end{prop}

The following statement holds (see \cite{DH}):

\begin{prop}[Theorem4.1 \protect\cite{DH}]\label{p42} A locally Lipschitz function $f:X\longrightarrow{\mathbb{R}}\cup\{\infty\}$ is strictly quasiconvex, if and only if, $\partial^\uparrow  f$ is strictly quasimonotone.
\end{prop}

The next result is well-known, see for instance \cite{PQ}.

\begin{prop}[Theorem 4.1 \protect\cite{PQ}]\label{p43} A lower semicontinuous, radially continuous function $f:X\longrightarrow{\mathbb{R}}\cup\{\infty\}$ is pseudoconvex, if and only if, $\partial^\uparrow  f$ is pseudomonotone.
\end{prop}

A similar result concerning on the strict pseudomonotonicity of the subdifferential of a strictly pseudoconvex function was established in  \cite{PQ}.

\begin{prop}[Theorem 5.1 \protect\cite{PQ}]\label{p44} A locally Lipschitz function $f:X\longrightarrow{\mathbb{R}}\cup\{\infty\}$ is strictly pseudoconvex, if and only if, $\partial^\uparrow  f$ is strictly pseudomonotone.
\end{prop}

\begin{defn}
\label{d41} A real valued function $f:X\longrightarrow{{\mathbb{R}\cup\{\infty\}}}$ is said to be locally convex, (respectively, locally quasiconvex, locally strictly quasiconvex, locally pseudoconvex, locally strictly pseudoconvex) if every $x\in X$ possesses an open and convex  neighborhood $U_x$ such that the restriction of $f$ on $U_x$, $f|_{U_x}$ is convex, (respectively, quasiconvex, strictly quasiconvex, pseudoconvex, strictly pseudoconvex).
\end{defn}

\begin{rmk}\rm It is known, that local convexity and also local generalized convexity of a differentiable function in general  implies its global counterpart (see \cite{KaPiLa,La}). However, according to the next example  local quasiconvexity does not implies global quasiconvexity.
\end{rmk}

\begin{ex}\rm[Example 4.9 \cite{La}] Consider the function $F:\Real\longrightarrow\Real,$ $$F(x)=\ds\left\{\begin{array}{lll}
\ds-\frac{x^2}{2}-x,\,\mbox{if\,}\, x<-1,\\
\ds\frac12,\,\mbox{if\,}\, x\in[-1,1],\\
\ds-\frac{x^2}{2}+x,\,\mbox{if\,}\, x>1.\\
\end{array}\right. $$
Then it can easily be verified that $F$ is continuously differentiable and locally quasiconvex but is not quasiconvex globally.
\end{ex}

Assume now, that the set $D\subseteq \dom f$ is self segment dense in $\dom f.$ We define the local generalized convexity of $f$ on $D$ as follows.

We say that the function $f$ is:

\begin{enumerate}
\item[(1)] \emph{locally pseudoconvex} on $D$, if every $z\in D$ admits an open neighborhood $U_z$ such that for all $x,y\in U_z\cap D$ the following implication holds
\begin{equation*}
f(y)<f(x)\Rightarrow  \forall x^*\in \partial^\uparrow  f (x):\<x^*,y-x\>< 0,
\end{equation*}

\item[(2)] \emph{locally strictly pseudoconvex} on $D$, if every $z\in D$ admits an open neighborhood $U_z$ such that for all $x,y\in U_z\cap D$ the following implication holds
\begin{equation*}
f(y)\le f(x)\Rightarrow  \forall x^*\in \partial^\uparrow  f (x):\<x^*,y-x\>< 0,
\end{equation*}

\item[(3)] \emph{locally quasiconvex} on $D$, if every $z\in D$ admits an open neighborhood $U_z$ such that for all $x,y\in U_z\cap D$  and $t\in[0,1]$ such that $x+t(y-x)\in D$ one has the following
\begin{equation*}
f(x+t(y-x))\le \max\{f(x),f(y)\},
\end{equation*}

\item[(4)] \emph{locally strictly quasiconvex} on $D$, if every $z\in D$ admits an open neighborhood $U_z$ such that for all $x,y\in U_z\cap D$  and $t\in(0,1)$ such that $x+t(y-x)\in D$ one has the following
\begin{equation*}
f(x+t(y-x))< \max\{f(x),f(y)\}.
\end{equation*}
\end{enumerate}

The following lemma will be useful in the sequel.
\begin{lem}\label{l41} Let $f:X\longrightarrow{\mathbb{R}}\cup\{\infty\}$ be lower semicontinuous with convex domain, and let $D\subseteq\dom f$ self segment-dense in $\dom f.$ Assume that $f$ is locally strictly quasiconvex on $D$ and that the Clarke-Rockafellar subdifferential of $f,\, \partial^\uparrow  f$ is lower semicontinuous on $\dom f.$ Then $\partial^\uparrow  f$ is locally quasimonotone on $D$. If $f$ is also locally Lipschitz, then $\partial^\uparrow  f$ is locally strictly quasimonotone on $D$.
\end{lem}
\begin{proof} Let $z\in D$ and consider $U_z$ an open neighborhood of $z$ such that for all $x,y\in U_z\cap D$  and $t\in(0,1)$ with $x+t(y-x)\in D$ we have
$f(x+t(y-x))< \max\{f(x),f(y)\}.$ We show that $\partial^\uparrow  f$ is strictly quasimonotone on $U_z\cap D.$

Indeed, let $x,y\in U_z\cap D$ and assume that $\<x^*,y-x\>>0$ for some $x^*\in\partial^\uparrow f(x).$ Then $f^\uparrow (x,y-x)>0$, hence there exists $\e>0$ and the sequences $x_n\To x,\, t_n\searrow 0$ such that
$$\inf_{v\in B(y-x,\e)}\frac{f(x_n+t_n v)-f(x_n)}{t_n}>0.$$
Since $D$ is self segment-dense in $\dom f$ we can assume that $(x_n),(x_n+t_n(y-x_n))\subseteq D.$
If $n$ is big enough then $\|x_n-x\|<\e$, hence $y-x_n\in B(y-x,\e).$
Thus, $f(x_n+t_n(y-x_n)>f(x_n)$ and in virtue of  locally strict quasiconvexity of $f$ on $D$  we get $f(x_n)<f(y).$ Using the lower semicontinuity of $f$ we have
$$f(x)\le\liminf_{n\To\infty}f(x_n),$$
thus $f(x)\le f(y).$

Suppose that there exists $y^*\in\partial^\uparrow f(y)$ such that $\<y^*,y-x\><0.$ Then $\<y^*,x-y\>>0$ which leads to $f^\uparrow (y,x-y)>0$ and using the same arguments as before, we conclude that $f(y)\le f(x).$ Hence, we have $f(x)=f(y).$

By the continuity property of the duality pairing we obtain that there exists an open neighborhood $V$ of $y^*$ such that $\<v^*,x-y\>>0$ for all $v^*\in V.$ Since $\partial^\uparrow f$ is lower semicontinuous, there exists an open neighborhood $U$ of $y$ such that $\partial^\uparrow f(u)\cap V\neq\emptyset$ for all $u\in U\cap U_z.$ Let $u\in (x,y)\cap U\cap D.$ Then there exists $u^*\in \partial^\uparrow f(u)\cap V$ such that $\<u^*,x-y\>>0.$ Since $u=x+t(y-x)$ for some $t\in(0,1)$ we have
$\<x^*,u-x\>=t\<x^*,y-x\>>0$ and $\<u^*,x-u\>=t\<u^*,x-y\>>0$, hence by using the same arguments as in firs part of the proof, we conclude that $f(u)= f(x).$
But then $f(u)=\max\{f(x),f(y)\}$ which contradicts the fact that  $f$ is locally strictly quasiconvex on $D.$

We proved that $\partial^\uparrow f$ is locally quasimonotne on $D.$  It remained to show, in the case when $f$ is locally Lipschitz, that for all $x,y\in U_z\cap D$  there exists $u\in (x,y)\cap D$ such that
$\<u^*,y-x\>\neq 0$ for some $u^*\in \partial^\uparrow f(u).$ Let $x,y\in U_z\cap D.$ Assume that $f(x)\neq f(y)$. This can be assumed since, otherwise, in virtue of locally strict  quasiconvexity of $f$ on $D$ one can take $y'\in(x,y)\cap D$ such that $f(y') <\max \{f(x), f(y)\}\Rightarrow f(y')\neq f(x).$ Assume that $f(y)-f(x)>0$, the case $f(y)-f(x)<0$ can be treated similarly.

According to Lebourg mean value theorem (see \cite{C}), there exists $u\in(x,y)$ and $u^*\in\partial^\uparrow f(u)$ 
such that $\<u^*,y-x\>= f(y)-f(x)>0.$
\end{proof}

In what follows we provide, in a Banach space context, sufficient conditions for strict quasiconvexity of a locally strictly quasiconvex functions. 

\begin{thm}\label{t41} Let $f:X\longrightarrow{\mathbb{R}}\cup \{\infty\}$ be a locally Lipschitz  function, locally strictly quasiconvex on $D,$ where $D\subseteq \dom f$ is self segment-dense in $\dom f.$ If $\partial^\uparrow  f$ is lower semicontinuous and has the property, that $\<z^*,x-y\>\neq 0$ for all $z\in[x,y]\cap \dom f\setminus D,\, x,y\in \dom f,\, x\neq y,\, z^*\in \partial^\uparrow  f(z)$ then $f$ is globally quasiconvex on $X.$
\end{thm}
\begin{proof}   According to Lemma \ref{l41}, $\partial^\uparrow  f$ is locally strictly quasimonotone on $D.$ According to Theorem \ref{t32}, $\partial^\uparrow  f$ is strictly quasimonotone on $\dom f.$ The conclusion follows from Proposition \ref{p42}.
\end{proof}

Similar results to Theorem \ref{t41} hold for locally pseudoconvex, respectively locally strict pseudoconvex functions.

\begin{thm}\label{t42}  Let $f:X\longrightarrow{\mathbb{R}}\cup \{\infty\}$ be a locally Lipschitz function, locally pseudoconvex on $D,$ where $D\subseteq X$ is self segment-dense in $\dom f.$ If $\partial^\uparrow  f$ is lower semicontinuous and has the property, that $\<z^*,x-y\>\neq 0$ for all $z\in[x,y]\cap X\setminus D,\, x,y\in \dom f,\, x\neq y,\, z^*\in \partial^\uparrow  f(z)$ then $f$ is globally pseudoconvex on $X.$
\end{thm}
\begin{proof} Let $z\in D$ and consider $U_z$ an open neighborhood of $z$ such that for all $x,y\in U_z\cap D$   we have
$f(y)< f(x)\Rightarrow  \forall x^*\in \partial^\uparrow  f (x):\<x^*,y-x\>< 0.$ We show that $\partial^\uparrow  f$ is pseudomonotone on $U_z\cap D.$

First of all, observe that $f$ is quasiconvex on  $U_z\cap D.$ Indeed, let $x,y\in  U_z\cap D$ and assume that $f(u)>\max\{f(x),f(y)\}$ for some $u\in(x,y)\cap D.$ But then   $u=x+t(y-x)$ for some $t\in(0,1)$ and by the pseudoconvexity of $f$ on  $U_z\cap D$ we obtain that $\forall u^*\in \partial^\uparrow  f (u)$ we have $\<u^*,x-u\>< 0$ and  $\<u^*,y-u\>< 0$ or, equivalently
$\<u^*,-t(y-x)\>< 0$ and  $\<u^*,(1-t)(y-x)\>< 0$, impossible.

Assume $\partial^\uparrow  f$ is not pseudomonotone on $U_z\cap D$ that is, there exists $x,y\in U_z\cap D$  such that $\<x^*,y-x\>>0$ for some $x^*\in\partial^\uparrow f(x)$ and $\<y^*,y-x\>\le0$ for some $y^*\in\partial^\uparrow f(y).$ According to Remark \ref{r4} $\<x^*,y-x\>>0\Rightarrow f(y)\ge f(x),$
and $\<y^*,y-x\>\le0\Rightarrow f(y)\le f(x),$ hence $f(x)=f(y).$

Since $\<x^*,y-x\>>0$ we have  $f^\uparrow (x,y-x)>0$, hence there exists $\e>0$ and the sequences $x_n\To x,\, t_n\searrow 0$ such that
$$\inf_{v\in B(y-x,\e)}\frac{f(x_n+t_n v)-f(x_n)}{t_n}>0.$$
Since $D$ is self segment-dense in $\dom f$ we can assume that $(x_n),(x_n+t_n(y-x_n))\subseteq D.$
If $n$ is big enough then $\|x_n-x\|<\frac{\e}{2}$, hence $y_1-x_n\in B(y-x,\e)$ if $y_1\in B\left(y,\frac{\e}{2}\right).$
Thus, $f(x_n+t_n(y_1-x_n)>f(x_n)$ and in virtue of   local quasiconvexity of $f$ on $D$  we get $f(x_n)< f(y_1)$ for all $y_1\in B\left(y,\frac{\e}{2}\right)\cap D.$ The latter relation combined with the pseudoconvexity of $f$ on $U_z\cap D$ in particular shows that $0\not\in \partial^\uparrow f(y).$

 Using the continuity of $f$ we have
$$f(y)=f(x)=\lim_{n\To\infty}f(x_n)\le f(y_1),$$
 Hence, $f(y)\le f(y_1)$ for all   $y_1\in B\left(y,\frac{\e}{2}\right)\cap D$ which shows that $y$ is a local minimum on $U_z\cap D.$ We show that $y$ is a  minimum on $B\left(y,\frac{\e}{2}\right)\cap \dom f.$ Indeed, let $u\in B\left(y,\frac{\e}{2}\right)\cap \dom f.$ Since $D$ is  dense in $\dom f,$ there exists a sequence $u_n\in B\left(y,\frac{\e}{2}\right)\cap D,\, u_n\To u.$ Obviously $f(u_n)\ge f(y)$, and since $f$ is continuous we have  $f(u)=\lim_{n\To \infty}f(u_n)\ge f(y).$ Hence $y$ is a local minimum on $U_z\cap \dom f$, which implies  $0\in \partial^\uparrow f(y),$ contradiction.

Thus,  $\partial^\uparrow  f$ is locally pseudomonotone on $ D.$ According to Theorem \ref{t33},  $\partial^\uparrow  f$ is  pseudomonotone on $\dom f.$ In virtue of Proposition \ref{p43}, $f$ is pseudoconvex.
\end{proof}

Next we establish some conditions that ensure  that a locally strictly pseudoconvex function on a self segment-dense subset in its domain is strictly pseudoconvex.

\begin{thm}\label{t43} Let $f:X\longrightarrow{\mathbb{R}}\cup \{\infty\}$ be a  locally Lipschitz  function, locally strictly pseudoconvex on $D,$ where $D\subseteq X$ is self segment-dense in $\dom f.$ If $\partial^\uparrow  f$ is lower semicontinuous and has the property, that $\<z^*,x-y\>\neq 0$ for all $z\in[x,y]\cap X\setminus D,\, x,y\in \dom f,\, x\neq y,\, z^*\in \partial^\uparrow  f(z)$ then $f$ is globally strictly pseudoconvex on $X.$
\end{thm}
\begin{proof} Let $z\in D$ and consider $U_z$ an open neighborhood of $z$ such that for all $x,y\in U_z\cap D$   we have
$f(y)\le f(x)\Rightarrow  \forall x^*\in \partial^\uparrow  f (x):\<x^*,y-x\>< 0.$ We show that $\partial^\uparrow  f$ is strictly pseudomonotone on $U_z\cap D.$

Suppose the contrary, that is there exist $x,y\in U_z\cap D$ such that $\<x^*,y-x\>\ge0$ for some $x^*\in \partial^\uparrow  f(x)$ and $\<y^*,y-x\>\le0$ for some $y^*\in \partial^\uparrow  f(y).$

According to Remark \ref{r4}  $$\<x^*,y-x\>\ge0\Rightarrow f(y)>f(x),$$ while  $$\<y^*,y-x\>\le0\Rightarrow f(x)>f(y),$$
impossible.

Since  $\partial^\uparrow  f$ is strictly  pseudomonotone on $D,$ according to Theorem \ref{t34},  $\partial^\uparrow  f$ is strictly  pseudomonotone on $\dom f.$ In virtue of Proposition \ref{p44}, $f$ is strictly pseudoconvex.
\end{proof}

\begin{rmk}\rm The assumptions imposed on $\partial^\uparrow  f$  in the hypothesis of the previous theorems cannot be dropped as the next example shows.
\end{rmk}
\begin{ex}\rm Let $f:\Real\To\Real,\, f(x)=-|x|.$ Then $f$ is locally strictly pseudoconvex on $\Real\setminus\{0\}$. We show that  $\partial^\uparrow  f$ is not lower semicontinuous in $x=0$ and  and that $f$ is not even quasiconvex on $\Real.$
\end{ex}
Indeed, it can easily be verified that
 $$\partial^\uparrow f(x)=\ds\left\{\begin{array}{lll}
1,\,\mbox{if\,}\, x<0,\\
\, [-1,1],\,\mbox{if\,}\, x=0,\\
-1,\,\mbox{if\,}\, x>0.\\
\end{array}\right. $$
Obviously $\partial^\uparrow f$ is not lower semicontinuous. Let $x=-1,\,y=1.$ Then for all $z\in (x,y)$ one has $f(z)>-1=\max\{f(x),f(y)\}$, which shows that $f$ is not quasiconvex.


\begin{thebibliography}{99}
\bibitem{A} D. Aussel, \emph{Subdifferential properties of quasiconvex and pseudoconvex functions: Unified approach}, J. Optimiz Theory App.
  \textbf{97}, 29-45 (1998)

\bibitem{ACL} D. Aussel, J.N. Corvellec, M. Lassonde, \emph{Subdifferential characterization of quasiconvexity and convexity}, J. Convex Anal. {\bf  1}, 195-201  (1994)

\bibitem{AF} J.P. Aubin, H. Frankowska, Set-valued analysis, Birkh\"auser, Boston (1990)

\bibitem{B}   H. Br\'ezis, Functional Analysis, Sobolev Spaces and Partial Differential Equations, Springer (2010)



\bibitem{C} F.H. Clarke, {Optimization and nonsmoth analysis}, Wiley, New York (1983)

\bibitem{DH} A. Daniilidis, N. Hadjisavvas, \emph{Characterization of nonsmooth semistrictly quasiconvex and strictly quasiconvex functions},
J. Optimiz Theory App. \textbf{ 102}, 525-536 (1999)

\bibitem{DH1} A. Daniilidis,  N. Hadjisavvas,\emph{On the subdifferentials of quasiconvex and pseudoconvex functions and cyclic monotonicity}, J. Math Anal Appl. \textbf{237}, 30-42 (1999)

\bibitem{HKS} N. Hadjisavvas, \emph{Generalized convexity, generalized monotonicity and nonsmooth analysis}, in Handbook of Generalized Convexity
and Generalized Monotonicity, N. Hadjisavas, S. Koml\'osi, and S. Schaible, eds., Springer Series Nonconvex Optimization and its Applications, Springer,
New York (2005)

\bibitem{KaPiLa} G. Kassay, C. Pintea, S. L\'aszl\'o, \emph{Monotone operators and first category sets}, Positivity \textbf{16}, 565-577 (2012)

\bibitem{KaPiLa1} G. Kassay, C. Pintea, S. L\'aszl\'o, \emph{Monotone operators and closed countable sets}, Optimization \textbf{60}, 1059-1069 (2011)

\bibitem{Ke} P. Kenderov,  \emph{Semi-continuity of set-valued mappings}, Fund. Math. \textbf{88}, 61-69 (1975)

\bibitem{Ko} S. Koml\'osi, \emph{Generalized Convexity and Generalized Derivatives},  Handbook of Generalized Convexity
and Generalized Monotonicity, N. Hadjisavas, S. Koml\'osi, and S. Schaible, eds., Springer Series Nonconvex Optimization and its Applications, Springer,
New York (2005)

\bibitem{La} S. L\'aszl\'o, \emph{\ Generalized monotone operators, generalized convex functions and closed countable sets}, J. Convex Anal. \textbf{18}, 1075-1091 (2011)

\bibitem{LYAK} L.J. Lin, M.F.Yang, Q.H. Ansari, G. Kassay, \emph{Existence results for Stampacchia and Minty type implicit variational inequalities
with multivalued maps}, Nonlinear Anal-Theor.  \textbf{ 61}, 1-19 (2005)

\bibitem{DTL} D.T. Luc, \emph{Existence Results for Densely Pseudomonotone Variational Inequalities}, J. Math Anal Appl. \textbf{254}, 291-308 (2001)

\bibitem{L-S} D.T. Luc, S. Schaible, \emph{Generalized Monotone Nonsmooth Maps}, J. Convex Anal. \textbf{3}, 195-205 (1996)

\bibitem{P} J.P. PENOT, \emph{Generalized convex functions in the light of non smooth analysis}, Lecture notes in Economics and Math. Systems \textbf{429}, Springer Verlag, 269–291 (1995)

\bibitem{P1}  J.P. PENOT, \emph{Are generalized derivatives useful for generalized convex functions?}, in Generalized Convexity, Generalized Monotonicity: Recent results, Crouzeix (Eds.), 3–59 (1998)

\bibitem{PQ} J.P. Penot, P. H. Quang, \emph{Generalized convexity of functions and generalized monotonicity of set-valued maps}, J. Optimiz Theory App. \textbf{92}, 343-356 (1997)

\bibitem{R} R.T. Rockafellar, \emph{Generalized directional derivatives and subgradients of nonconvex functions}, Canadian Journal of Mathematics \textbf{32},
257-280 (1980)
\end{thebibliography}
\end{document}